\documentclass{amsart}

\usepackage{amsmath,amsthm,amsfonts,amssymb,amscd,graphicx} 
\usepackage{enumerate}
\usepackage{hyperref}
\usepackage{color}
\usepackage[all]{xy}

\newtheorem{theorem}{Theorem}[section]
\newtheorem{lemma}[theorem]{Lemma}
\newtheorem{corollary}[theorem]{Corollary}
\newtheorem{proposition}[theorem]{Proposition}

\theoremstyle{definition}

\theoremstyle{remark}
\newtheorem{remark}[theorem]{Remark}

\numberwithin{equation}{subsection}

\begin{document}

	\date{\today}
	\title[On the length of cohomology spheres]{On the length of cohomology spheres}
	
	\author[D. de Mattos]{Denise de Mattos}
	\address{\textup{Denise de Mattos:}
		Universidade de S\~{a}o Paulo\\
		Instituto de Ci\^{e}ncias Matem\'{a}ticas e de Computa\c{c}\~{a}o\\
		Avenida Trabalhador S\~{a}o-carlense, 400 - Centro\\
		CEP: 13566-590, S\~{a}o Carlos, SP Brazil} 
	\email{deniseml@icmc.usp.br}
	
	\author[E. dos Santos]{Edivaldo Lopes dos Santos}
	\address{\textup{Edivaldo Lopes dos Santos:}
		Universidade Federal de S\~{a}o Carlos\\
		Departamento de Matem\'{a}tica\\
		Km 235, Rodovia Washington Lu\'{i}s - Jardim Guanabara\\
		CEP: 13565-905 S\~{a}o Carlos, SP Brazil} 
	\email{edivaldo@dm.ufscar.br}
	
	\author[N. Silva]{Nelson Antonio Silva}
	\address{\textup{Nelson Antonio Silva:}
		Universidade Federal de Lavras\\
		Departamento de Ci\^{e}ncias Exatas\\
		Av. Doutor Sylvio Menicucci, 1001 - Kennedy\\
		CEP: 37200-000 Lavras, MG Brazil} 
	\email{nelson.silva@ufla.br}

	\keywords{Cohomological length, Cohomology Spheres, Borsuk-Ulam theorem, Bourgin-Yang theorem,
	equivariant map} \subjclass[2010]{Primary 55M10, 55M30;\; Secondary55N91}

	\begin{abstract}
		In \cite{Bartsch}, T. Bartsch 
		provided detailed and broad exposition of a numerical cohomological
		index theory for $ G $-spaces, known as the {length}, where $ G $ is a compact Lie
		group. We present the {length} of $ G $-spaces which are cohomology spheres and
		$ G $ is a $p$-torus or a torus group, where $p$ is a prime.
		As a consequence, we obtain Borsuk-Ulam and Bourgin-Yang type theorems 
		in this context. A sharper version of the Bourgin-Yang theorem 
		for topological manifolds is also proved. Also, we give some  general 
		results regarding the upper and lower bound for the length. 
	\end{abstract}

	\maketitle

	\section{Introduction}\label{section: introduction}
	Let $G$ be a compact Lie group. In \cite{Bartsch}, Thomas Bartsch discuss in
	details the properties and results of a	numerical cohomological index theory known as 
	$(\mathcal{A}, h^*, I)$-length, or simply, the length $\ell$ (Definition \ref{definition: l_H}). 
	Like all index theories, the length is a powerful tool in the study of equivariant maps
	and, in recent years, it has been used to prove different versions of Borsuk-Ulam and 
	Bourgin-Yang type theorems \cite{BMS,DEM,DEM1}.
	
	In the study of critical points of functions with symmetry, Bartsch and Clapp \cite{BartschClapp}
	computed the value of the length of representation spheres $S(V)$ 
	\cite[Proposition 2.4]{BartschClapp} for  $ p $-torus  
	or torus groups, i.e., $G=(\mathbb{Z}_p)^k$ or $(S^1)^k$, where $k\geq 1$ and $p$ prime. 
	Considering such groups, we present a more general result by 
	providing the length of the pair $(X,X^G)$,
	where  $X$ is a compact $G$-space that is a cohomology spheres and  
	$X^G$ is the fixed point set (Theorem \ref{theorem: length_cohom_sphere}). This is achieved by considering the
	splitting principle in the Euler class $e(X,X^G)$ of the oriented pair $(X,X^G)$
	associated with the fibration $X_G\to BG$ given by the Borel construction 
	\cite{Hsiang,Dieck}.

	As an immediate consequence of monotonicity of the length, a Borsuk-Ulam type
	result is obtained (Corollary \ref{theorem: bu}). This provides necessary conditions on the existence
	of equivariant maps between cohomology spheres under the actions of $p$-torus or torus. 
	Motivated by results in \cite{BMMS,DEM}, regarding the sufficient conditions
	to the Borsuk-Ulam theorem, we provide a result for equivariant maps between $G$-ANR spaces
	which are cohomology spheres and representation spheres in the case $G=(\mathbb{Z}_p)^k$
	(Theorem \ref{theorem: bureverse}).
	We remark that an alternative proof could be done by using the results provided by
	\cite[Chapter II]{Dieck} and nicely presented in \cite[Theorem 3.2]{BMMS} for this
	context. Our proof rely on the calculation of the equivariant 
	Lusternik-Schnirelmann $\mathcal{A}$-cat and $\mathcal{A}$-genus \cite[Definitions 2.6 and 2.8]{Bartsch} of such cohomology spheres. As a corollary we conclude that the Euler class of a $(\mathbb{Z}_p)^k$-ANR cohomology sphere will be polynomial, where $p$ is an odd prime (Corollary \ref{corollary: e_polynomial}).
	
	Yang\cite{Yang,Yang2} and, independently, Bourgin\cite{Bourgin}
	proved that if $ f\colon S^{n-1}\to \mathbb{R}^m $ is a $\mathbb{Z}_2$-equivariant map 
	then $\dim Z_f\geq n-m-1$, where $ Z_f=f^{-1}(0) $ and $ ``\textup{dim}" $ stands for
	covering dimension (this is the so called Bourgin-Yang  theorem). 
	Consequently, if $ n>m $, then $ \dim Z_f\neq \emptyset $. 
	Hence, there is no	$ \mathbb{Z}_2 $-equivariant map $ S^n\to S^{m} $ with respect to the
	antipodal action which implies the classical Borsuk-Ulam theorem \cite{Borsuk}. 	
	B\l{}aszczyk et. al. \cite{BMS} presented abstract and general 
	versions of the Bourgin-Yang theorem in different settings by  making use
	of the length. These results depend on estimations of lower- and upper-bounds for $\ell$.
	In this sense, we present some general estimations. Namely, we give an upper-bound for
	the length (Theorem \ref{theorem: length_ubound}) of compact $G$-spaces with no 
	fixed points in terms of covering dimension where 
	$G$ is any compact Lie group. This allows us to give a certain type of Bourgin-Yang theorem
	considering the $p$-torus and torus groups and recover some classical results in particular
	cases. In spite the fact this is a rough estimation, 
	we remark is the best one could get with 
	our choices for the definition of length (Remark \ref{remark: dim_best}).
	In the case the domain $ X $ is a topological closed orientable manifold and 
	$ n $-acyclic over the field corresponding to the $p$-torus or torus group, 
	and using a totally different
	technique, following \cite[Theorem 2.1]{DEM}, a version of this theorem with optimal
	estimate is obtained.
	
	At last, we remark that the length can be used to obtain results 
	related with the classical Borel formula (Theorem \ref{theorem: borel_euler}). 
	We present an lower-bound for the length
	of any compact $G$-spaces in terms of the length of fixed points set $X^H$,
	where $H$ are subtorus of rank $k-1$ of $p$-torus or torus group $G$ of rank $k$ (Theorem \ref{theorem: length_lbound}).

\section{Preliminaries}\label{section: preliminaries}

	Let $G$ be a $p$-torus or a torus group of rank $k\geq 1$. We will distinguish the cases
	by the following: $G=(\mathbb{Z}_p)^k$ for $p\geq 2$  or
	$G=(S^1)^k$ for $p=0$.
	A $p$-subtorus of rank $t\leq k$ of $G$ will be a subgroup $(S^1)^t$ (for $p=0$) or  $(\mathbb{Z}_p)^t$
	(for $p\geq 2$).
	We consider the category of paracompact G-pairs $(X,A)$, where $X$ is a paracompact Hausdorff
	space and $A$ is a closed subspace. 
	The isotropy subgroup of $x\in X$ is $G_x=\{x\in X| gx=x,\,\forall g\in G\}$ and
	the orbit of $x$ is the $G$-subspace $G(x)=\{gx|\,g\in G\}\}\cong G/G_x$. 
	The orbit space of the $G$-space $X$ will be denoted by $X/G$.
	For any closed subgroup $H$ of $G$,  $X^H=\{x\in X| hx=x,\forall h\in H\}$ is the set of fixed points in $X$ by the induced $H$-action. 

	Let $X_G=(EG\times X)/G$ be the Borel space where is $EG$ is the total space of
	the universal principal  $G$-bundle $EG\to BG$ and $BG=EG/G$ is the universal classifying space.
	For a $G$-pair $(X,A)$, we denote $H_G^*(X,A;\mathbb{F}) = H^*(X_G,A_G;\mathbb{F})$, 
	the Borel equivariant cohomology, where $H^*$  will always be the \v{C}ech cohomology
	and $\mathbb{F}=\mathbb{Z}_p$ or $\mathbb{Q}$ whether $p\geq 2$ or $p=0$ . The map $p_X^*\colon H^*(BG;\mathbb{F})\to H_G^*(X;\mathbb{F})$, 
	induced by $X\to \{\rm pt\}$, gives a $H^*(BG;\mathbb{F})$-module
	structure on $H_G^*(X;\mathbb{F})$ (also on $H_G^*(X,A;\mathbb{F})$) by $xy:=p_X^*(x)\cup y\in H_G^{m+n}(X;\mathbb{F})$,
	for $x\in H^m(BG;\mathbb{F})$ and $y\in H_G^n(X;\mathbb{F})$. We mainly deal with the following cohomology rings
	$H^*(BG;\mathbb{F})$.
	
	\begin{itemize}
		\item If  $ G=(\mathbb{Z}_2)^k $ then
		$ H^*(BG;\mathbb{Z}_2)\cong \mathbb{Z}_2[t_1,\ldots,t_k] $, 
		where $ t_i\in H^1(BG;\mathbb{Z}_2) $.
		\item If $ G=(\mathbb{Z}_p)^k $, then
		$ H^*(BG;\mathbb{Z}_p)\cong \mathbb{Z}_p[t_1,\ldots,t_k]\otimes_{\mathbb{Z}_p}
		\Lambda(s_1,\ldots,s_k) $, 	where $p>2$, $ t_i\in H^2(BG;\mathbb{Z}_p) $ and
		$ s_i\in H^1(BG;\mathbb{Z}_p) $.
		\item If $ G=(S^1)^k $, then
		$ H^*(BG;\mathbb{Q})\cong \mathbb{Q}[t_1,\ldots,t_k] $, 
		where $ t_i\in H^2(BG;\mathbb{Q}) $.
	\end{itemize}

	In the case $p>2$, we set $P^*(G)=\mathbb{Z_p}[t_1,\ldots,t_k]$ the polynomial
	part of $H^*(BG;\mathbb{Z}_p)$. Since we mainly deal
	with $p$-tori groups, we shall suppress the coefficient field $\mathbb{F}$ and keep the choices
	as above. 
	
	For any subtorus $H$ of $G$, we have $H_G^*(G/H)\cong H^*(BH)$. In the particular case
	$p=0$ (or $2$) and $H$ is a subtorus of rank $k-1$,  the kernel of the map $H^*(BG)\to H^*(BH)$,
	induced by	$G/H\to G/G$, is a principal ideal $(s_H)$, where $s_H\in H^2(BG)$ 
	(or $s_H\in H^1(BG)$). 
	For $p>2$, $P^*(G)\cap \ker[H^*(BG)\to H^*(BH)]=(s_H)$
	and $s_H\in H^2(BG)$.

%
\subsection{The length.}\label{definition: l_H}
	Fix a set $\mathcal{A}$ of $G$-spaces and $I$ an ideal of the cohomology ring $H^*(BG)$. 
	Let $(X,A)$ be a $G$-pair.
	The $(\mathcal{A}, H_G^*,I)\textup{-length} $
	of $ (X,A) $ is the smallest integer 
	$ \lambda\geq 0 $ such that there exist $ A_1,\ldots,A_{\lambda}\in \mathcal{A}$
	that for any $\omega_i\in I\cap \textup{ker}\left[ 
	H^*(BG)\longrightarrow H_G^*(A_i)\right]$, $1\leq i\leq \lambda$, 
	we have $\omega_1\cdots\omega_\lambda\cdot \gamma = 0\in H_G^*(X,A)$,
	for all $ \gamma\in H_G^*(X,A) $. If such $\lambda$ does not exist, we write
	that $(\mathcal{A}, H_G^*,I)\textup{-length} $ of $(X,A)$ is $\infty$.

	We shall make standard choices for $\mathcal{A}$ and $I$:
	\begin{itemize}
		\item $\mathcal{A} =\{G/H;\,H\subsetneq G\textup{ closed subgroup}\}$. This is
		equivalent, in terms of the value of the length, 
		to $\mathcal{A}'=\{G/H;\,H\textrm{ has rank }k-1\}$
		\cite[Observation 5.5]{Bartsch};
		\item For $p=0,2$, $I=H^*(BG)$ and for $p>2$, $I=P^*(G)$ the polynomial part.
		\item For a subtorus $H$ of rank $k-1$, we set $\mathcal{A}_H=\{G/H\}$
		and $I_H=(s_H)$.
	\end{itemize}

	We simply write $\ell$ (or $\ell_H)$ instead of $(\mathcal{A},H_G^*,I)$-length 
	(or $(\mathcal{A}_H,H_{G/H}^*,I_H)$-length). Also, in the case $A=\emptyset$, we write
	$\ell(X)$ (or $\ell_H(X)$) instead of $\ell(X,\emptyset)$ (or $\ell_H(X,\emptyset)$). 

	\begin{proposition}[{\cite[Proposition 4.7 and Corollary 4.9]{Bartsch}}]\label{proposition: length_prop}
		Let $X$, $Y$ be two
		$G$-spaces:
		\begin{itemize}
			\item[i)] If $f\colon X\to Y$ is a $G$-equivariant map, then 
				  $\ell(X)\leq \ell(Y)$.
			\item[ii)] $\ell(X)\leq\mathcal{A}\textrm{-genus}(X)$, where the $\mathcal{A}$-genus of
			$X$ is the least integer $ t\geq 0 $ such that there exists	a $ G $-equivariant map 
			$ X\longrightarrow A_1\star \cdots \star A_t $, where $ A_i\in\mathcal{A} $ for $ i=1,\ldots,t$ and  $``\star"$ means the join operation.
			\item[iii)] If $X$ is a compact $G$-space such that $X^G=\emptyset$, then 
			$\ell(X)<\infty.$
		\end{itemize}
	\end{proposition}
	
	\begin{remark}\label{remark: len_gen_cat}
		Here we only mentioned the results and	choices	on the definition of the length 
		we are going to use. The length is defined in a very broad scenario 
		and has many properties \cite[Chapter 4]{Bartsch}.  
		The given definition of $\mathcal{A}$-genus
		is one of its characterizations. It can be seen as a particular case of the definition
		of the equivariant Lusternik-Schnirelmann category $\mathcal{A}$-cat
		\cite[Definition 2.6]{Bartsch}. 
		We have the inequality
		$\mathcal{A}$-genus$(X)\leq\mathcal{A}$-cat$(X)$ \cite[Proposition 2.10]{Bartsch}. 
	\end{remark}

%
\subsection{Cohomology spheres and Euler classes.}
	A compact Hausdorff space $ X $ is a  $(\textup{mod}\,{p} )$-cohomology $n$-sphere
	when $ H^*(X;\mathbb{F})\cong H^*(S^n;\mathbb{F})$, where $\mathbb{F}=\mathbb{Z}_p$ or $\mathbb{Q}$ 
	depending on whether $p\geq 2$
	or $p=0$. For $G=(\mathbb{Z}_p)^k$ or $(S^1)^k$, 
	the classical Smith theorem states that $X^G$ is a $(\textup{mod}\,{p})$-cohomology
	$r$-sphere, where $-1\leq r\leq n$, and $r=-1$ when $X^G=\emptyset$. 
	Let $e(X,X^G)$ be the Euler class of the oriented pair $(X,X^G)$ as defined in
	\cite[Chapter III, 4.25]{Dieck}. We shall make use of the results:
	
	\begin{theorem}\label{theorem: borel_euler}
		Consider $G$ and $X$ as above:
		\begin{itemize}
			\item[i)] $\left( \textrm{\cite[Chapter III, page 205]{Dieck}}\right)$. 
			There exists a
			 $H^*(BG)$-isomorphism \newline $H_G^*(X,X^G)\cong H^*(BG)/(e(X,X^G )) $.
			\item[ii)]$\left( \textrm{\cite[Chapter III, Theorem 4.40]{Dieck}}\right)$ Borel Formula.
			 Let $\mathcal{H} = \{H\subset G;\, H\textup{ has }$ $\textup{ rank } k-1\}$
			then $n-r = \sum_{H\in\mathcal{H}}n(H)-r$, where $n(H)$
			is the dimension of the cohomology sphere $X^H$.
		\end{itemize}
	\end{theorem}
	
	\begin{remark}\label{remark: annihilator}
		Since $H_G^*(X,X^G)\cong \frac{H^*(BG)}{(e)}$, where $e=e(X,X^G)$, we have
		that the annihilator set of $H^*_G(X,X^G)$ is generated by $(e)$, i.e.,
		$A=\{a\in H^*(BG)|\,a\gamma=0\,\textup{ for all }\gamma\in H^*_G(X,X^G)\}= (e)$.
		In the particular case, $G=(\mathbb{Z}_p)^k$ and $e$ is a polynomial element in
		$P^*(G)\subset H^*(BG)$, we have $A\cap P^*(G)=(e)$. 
		If $e$ is not polynomial,
		we  can write $e=z+n$, where $n$ stands for the nilpotent part, and we have
		$z^2=e(z-n)\in A\cap P^*(G)$.
	\end{remark}
%
\section{The length of cohomology spheres}\label{section: length_cohomology_spheres}

	\begin{theorem}\label{theorem: length_cohom_sphere} 
		Let $G=(\mathbb{Z}_p)^k$ or $(S^1)^k$,
		$ X $ be a ($ \textup{mod}\,p $)-cohomology $n$-sphere and $r<n$ be the dimension of
		the $(\textup{mod}\,p)$-cohomology sphere $ X^G $, then:
		 
		 \vskip-.5cm
		 $$\ell (X,X^G)= 
		 \begin{cases} 
		 	n-r, \,\,\displaystyle\textrm{ if } p=2,\\
		 	\displaystyle\frac{n-r}{2}, \textrm{ if }  p=0 \textrm{ or }
		 	p>2 \textrm{ and } e(X,X^G)   \textrm{ is polynomial.}
		 \end{cases}$$
		 For $p>2$, in the case $e(X,X^G)$ is not polynomial, we have 
		 $\frac{n+1}{2}\leq\ell(X,X^G)\leq n+1$.
	\end{theorem}

	\begin{proof}
		All the cases are proved in the same way considering that the Euler class is polynomial
		when $p>2$.
		
		Essentially, the length of $(X,X^G)$ resumes in finding 
		an element $ \omega\in H^*(BG)$	that annihilates $H^*_G(X,X^G)$ 
		and can be written as a product of generators 
		of certain ideals $ I\cap \ker [H^*(BG)\to H^*(BH)] $, for homogeneous 
		spaces $ G/H\in \mathcal{A}$, where  the number of factors in this
		product is the least as possible. By Remark \ref{remark: annihilator},
		this is the same as to find such $\omega$ that is in the ideal ring $(e)$.
		
		This is already the case for the Euler class $e$.
		It is a well-known fact \cite[Chapter III, Section 4]{Dieck} that 
		$e = (s_{H_1})^{k_1}\cdots (s_{H_t})^{k_t} $
		where $ (s_{H_i}) =I\cap \ker [H^*(BG;R)\to H^*(BH_i;R)]$
		and $ H_i $ are subtori of rank $ k-1 $ such that $X^{H_i}\neq\emptyset$.
		Here we have that $ k_i=n(H_i)-r $, for $p=2$, and 
		$k_i = \frac{n(H_i)-r}{2}$, for $p=0$ or $p>2$, where $n(H)$ is the
		dimension of the cohomology sphere $X^H$. 
		So, by choosing	$ k_i $ times $G/H_i \in \mathcal{A}$,
		$ i=1,\ldots, s $, for all 
		$ \omega^i_1,\ldots, \omega^i_{k_i}\in 
		(s_{H_i})\in I\cap\ker [H^*(BG)\to H^*(G/H_i)], $ with
		$ i=1,\ldots,t$, we have that
		$ \omega= \prod_{i=1}^{t} \omega^i_1\ldots \omega^i_{k_i}\in (e) $, which 
		implies that $\ell(X,X^G)=k_1+\ldots+k_t=\deg(e)$ and this completes the proof. 
		
		Now, if $p>2$ and $e(X,X^G)$ is polynomial, we should have that $e=z+n$, where
		$z=(s_{H_1})^{k_1}\cdots (s_{H_t})^{k_t}$ and $n$ is nilpotent. As mentioned in Remark
		\ref{remark: annihilator}, $z^2\in(e)$, then $\ell(X,X^G)\leq n-r$. Clearly,
		$\ell(X,X^G)\geq \frac{n-r}{2}$ because $\deg(e)=n-r$.
		\end{proof}
		
	The next remark points out some situations where the Euler class $e$ is polynomial for $p>2$.
	
	\begin{remark}\label{remark: polynomial} 
		If we take a subtorus $H$ of rank $k-1$ such that $X^H\neq \emptyset$, 
		we should have that $\ell(X^H,X^G)=\ell_H(X^H,X^G)=\frac{n(H)-r}{2}$, where 
		$\ell_H$ is as in \ref{definition: l_H}. 
		Indeed, in the case, $e(X^H,X^G)$ is equal to $(s_H)^{n(H)}$ which is polynomial
		\cite[Theorem 4.40]{Dieck}. So in the particular case that $G=\mathbb{Z}_p$ 
		and $X^G=\emptyset$, we should have $e$ polynomial and $\ell(X)=\frac{n+1}{2}$. 
		We shall see in Corollary \ref{corollary: e_polynomial} that when	
		$X$ is a $G$-ANR space, we also have that $e$ is polynomial and 
		$\ell(X)=\frac{n+1}{2}$.
	\end{remark}
	\begin{remark}
		The Theorem $\ref{theorem: length_cohom_sphere}$ generalizes the result
		\cite[Proposition 2.4]{BartschClapp} given by Bartsch and Clapp in the context of
		representation spheres.
		In \cite[Proposition 3.6]{BMS}, the authors give a lower bound for a $G$-space $X$ compact
		(or paracompact with finite covering dimension) such that $H^i(X)=0$, when $0<i<n$. 
		Namely, $\ell(X)\geq n+1$ for $p=2$ and $\ell(X)\geq\frac{n+1}{2}$, otherwise.
	\end{remark}

	We derive the following version
	of the Borsuk-Ulam theorem.
	
	\begin{corollary}[Borsuk-Ulam]\label{theorem: bu}
		Let $G=(\mathbb{Z}_p)^k$ or $(S^1)^k$ and $X, Y$ two $(\textup{mod}\,p )$-cohomology
		spheres of dimension $n$ and $m$, respectively.	Suppose that $X^G=Y^G=\emptyset$. 
		If there is a $G$-equivariant map $f\colon X\to Y$, then 
		$\dim X^H\leq \dim Y^H, \;\;  \textup{for all} \;\;  H<G \;\;\textup{of rank} \;\;k-1$. 
		In particular, $ n\leq m $. Thus, if $ n>m $, there is no $ G $-equivariant map from 
		$X$ to $Y$. 
	\end{corollary}
	\begin{proof}
		Suppose that exists a $G$-equivariant map $f\colon X\to Y$. Note that
		$ f^H:=f|_{X^H}\colon X^H\to Y^H $ is
		$ G/H $-equivariant map, for any $ H<G $. When
		$ H $ is a subtorus of rank $ k-1 $, from monotonicity of the length 
		(Proposition \ref{proposition: length_prop}(i)) and 
		Theorem \ref{theorem: length_cohom_sphere} (including Remark 
		\ref{remark: polynomial})),	we have that 
		$\ell_H(X^H)=\frac{\dim(X^H)+1}{2} \leq \ell_H(Y^H)=\frac{\dim(Y^H)+1}{2}$. This
		proves the first part.
		
		Now, suppose $ n>m $. By the Borel formula, 
		$ n+1=\sum_H (\dim X^H +1) > \sum_H (\dim Y^H +1) = m+1$, then
		we should have $ \dim X^H > \dim Y^H$
		for at least one subtorus $ H $ of rank $ k-1 $. Thus,
		there is no $ G $-equivariant map from $ X $ to $ Y $.
	\end{proof}

	\begin{remark} 
		For the conclusion, one could extends the result  to spaces $X$ such that
		$H^i(X)=0$, for $0<i<n$, by considering the estimations of
		\cite[Proposition 3.6]{BMS}.
	\end{remark}

%
\subsection{A converse for the Borsuk-Ulam}\label{section: converse-bu}
	Here $G=(\mathbb{Z}_p)^k$, for $p\geq 2$.
	We will prove a certain converse for the Borsuk-Ulam theorem,
	finding sufficient conditions for the existence of
	$ G $-equivariant maps between a $(\textup{mod}\,p)$-cohomology sphere $ X $,
	that is also a $G$-ANR space, and a representation sphere $ S(V) $ of same  dimension.
	For that we use the $\mathcal{A}$-genus of $X$. As a corollary, we show that the Euler class
	$e(X)=e(X,\emptyset)$ is always polynomial, when $p>2$.

	\begin{lemma}\label{lemma: cat_genus_GANR}
		Let $G= (\mathbb{Z}_p)^k$ and suppose that $X$ is a $(\textup{mod}\,p)$-cohomology $n$-sphere  and $ G $-ANR space. Then 
		$\mathcal{A}\textup{-cat}(X)=\mathcal{A}\textup{-genus}(X)=n+1$.
	\end{lemma}
	\begin{proof}
		From \cite[Chapter 2]{Bartsch} and  \cite[Proposition 3.7]{Deo} we have that 
		$\mathcal{A}\textup{-cat}(X)\leq (\dim(X)+1)\cdot 
		{\max_{H\subset G}}\, c(H)$, 	
		where $ c(H) $ is the number of connected components of $X^H/NH $ 
		and $ H \subset G $ are closed subgroups. Here $``\dim"$ stands for covering dimension. 
		Since $c(H)=1$ \cite[Lemma 2.2]{MarzIzy2}, then 
		$\mathcal{A}\textup{-cat}(X)\leq \dim(X)+1=n+1$.		
		
		As pointed out in Remark \ref{remark: len_gen_cat}, $\ell(X)\leq
		\mathcal{A}\textup{-cat}(X)\leq\mathcal{A}\textup{-genus}(X)$. We shall verify then
		that $n+1\leq \mathcal{A}\textup{-genus}(X)$.
		For the case $ p=2 $, from Proposition \ref{proposition: length_prop} ii), $\mathcal{A}\textup{-genus}(X)\geq n+1$ since
		we have $ \ell(X) =n+1$. 
		
		Let us analyze now the case $ p>2 $.
		Suppose that $ \mathcal{A}\textup{-genus}(X)=t $, then there is a 
		$ G $-equivariant map $ X\longrightarrow G/H_1\star \cdots \star G/H_t$, where
		$G/H_i\cong\mathbb{Z}_p$. Thus we have a  $ G $-equivariant map 
		$ \varphi\colon X\star X\longrightarrow (G/H_1\star G/H_1)\star \cdots \star 
		(G/H_k\star G/H_t)$. By a well-known trick
		\cite[Remark 5.6]{Bartsch}, we can map $G/H_i\star G/H_i$ to $S(V_i)$
		equivariantly, where $ V_i $ is an irreducible representation given by the character 
		$ G/H_i\hookrightarrow	S^1$. 		
		Let $ W=V_1\oplus\cdots\oplus V_t $. Note that $ W $ is a vector space of 
		real dimension $ 2t $, then $ SW $ has  $ 2t-1 $ dimension. The map
		$ \varphi $ induces a  $ G $-equivariant map between $ X\star X $
		and $ S(W) $.
		Since $X\star X$ is a cohomology sphere of dimension $ 2n + 1 $ and
		$\ell(SW)=t$ \cite[Proposition 2.4]{BartschClapp},
		we see that $ n+1\leq \ell(X\star X) \leq \ell (S(W))=t $, where
		the first inequality is given by Theorem \ref{theorem: length_cohom_sphere}. 
		Thus, $ n + 1\leq \mathcal{A}\textup{-genus}(X)\leq \mathcal{A}\textup{-genus}(X).$
	\end{proof}

	\begin{remark}\label{remark: existence1} 
		For any subtorus $H$ of rank $k-1$ of $G= (\mathbb{Z}_p)^k$
		such that $X^G=\emptyset$, if we consider $\mathcal{A}_H=\{G/H\}$, then
		$\mathcal{A}_H\textup{-genus}(X^H)=n(H)+1$, where 
		$ n(H) $ is the dimension of the cohomology sphere $ X^H $.
	\end{remark}

	\begin{lemma}[Existence]\label{lemma: existence}
		Let $ G = (\mathbb{Z}_p)^k $ and $ X $ be a $(\textup{mod}\,p)$-cohomology $n$-sphere 
		such that $ X$ is a $ G $-ANR space and  $X^G=\emptyset $. Then there exists a
		$ G $-equivariant map between $ X $ and a representation sphere 
		of same dimension.
	\end{lemma}
	\begin{proof} It follows from Corollary \ref{lemma: cat_genus_GANR} that
		$\mathcal{A}\textup{-genus}(X)=n+1 $ and, by definition (Proposition \ref{proposition: length_prop}(ii)), 
		there exists a $G$-equivariant map 
		$ f\colon X\to G/K_1\star \cdots\star G/K_{n+1}$, where
		$G/K_i\in\mathcal{A}$. For $ p=2 $, the result follows already because 
		$G/K_i\cong \mathbb{Z}_2\cong S^0$ and then the $G$-equivariant map is 
		$ X\to S^0\star\cdots\star S^0\cong SV $, where $ \dim SV=n$.
	
		Suppose $ p>2 $. For every subtorus $H$ of rank $k-1$ such that $X^H\neq\emptyset$,
		we have the $G/H$-equivariant map
		$ f|_{X^H} \colon X^H\to (G/K_1\star \cdots\star G/K_{n+1})^H\cong
		G/H\star \cdots\star G/H$, where (by Remark \ref{remark: existence1} and Borel Formula) 
		the number of copies of $G/H$ in the join must be $n(H)+1$.
		
		Reordering and regrouping the factors of the join $G/K_1\star \cdots\star G/K_{n+1}$,
		we may write it as $M_1\star\cdots\star M_t$, where 
		$M_i\cong G/H_i\star \cdots\star G/H_i$ ($n(H_i)+1$ times). By the same trick
		\cite[Remark 5.6]{Bartsch} used in the previously lemma,
		every pair of join in $M_i$, $G/H_i\star G/H_i$, can be mapped equivariantly to
		$S(V_i)$, where $ V_i $ a irreducible representation given by the character 
		$ G/H_i\cong\mathbb{Z}_p\hookrightarrow S^1$. Since $n(H_i)+1$ is an even number,
		we should have $M_i\cong S(V_i)\oplus\cdots\oplus S(V_i)\cong S(V_{H_i})$, where 
		$\dim_{\mathbb{C}} S(V_{H_i})=n(H_i)$.
		
		Thus $f\colon X\to G/K_1\star \cdots\star G/K_{n+1}\cong
		S(V_{H_1})\star\cdots\star S(V_{H_t})\cong S(V)$, where $\dim_{\mathbb{C}} S(V)=
		\sum_i n(H_i)=n$.
	\end{proof}

	\begin{theorem} \label{theorem: bureverse}
		Let $ G = (\mathbb{Z}_p)^k $ and $ X $ be a $(\textup{mod}\,p)$-cohomology $n$-sphere 
		such that $ X$ is a $ G $-ANR space and  $X^G=\emptyset $.
		There exist a  $ G $-equivariant map between $ X $ and a representation sphere
		$S(V) $ of $ G $ with
		$ V^G=\{0\} $ if, and only if, $ \dim X^H \leq \dim SV^H$,
		for all $ H $ subtori of rank $ k-1 $ such that $ X^H\neq \emptyset $.
	\end{theorem}	

	\begin{proof}
		From Corollary \ref{theorem: bu} 
		we have a necessary condition for the existence of the map.
		Now suppose that $\dim X^H\leq  \dim_{\mathbb{C}} SV^H $, for all   $ H $.
		From Lemma \ref{lemma: existence} there exists a   $ G $-equivariant map 
		$ X\to SW $, where 
		$ \dim_{\mathbb{C}} SW=\dim X $ such that $ \dim X^H= \dim_{\mathbb{C}} SW^H $, 
		for all subtori $ H $
		of  rank $ k-1 $ with $ X^H\neq \emptyset $.
		By 	\cite[Teorema 2.5]{MAR1}, 
		there exists a   $ G $-equivariant map
		between $ SW $ and  $ SV $, thus there exists a  $ G $-equivariant map 
		between $ X $ and $ SV $.
	\end{proof}		

	\begin{remark}
		For an alternative proof, considering that a $G$-ANR space
		is $G$-homotopy equivalent to  a $G$-CW complex \cite[Theorem 13.3]{MUR}
		one could use the result given in \cite[Theorem 3.2]{BMMS}.
	\end{remark}

	As a consequence we obtain that the Euler class of a $(\textup{mod}\,p)$-cohomology sphere  
	$ X $ such that $ X^G=\emptyset $ will be
	polynomial when $ X $ is $ (\mathbb{Z}_p)^k $-ANR and $p>2$.
	
	\begin{corollary}\label{corollary: e_polynomial}
		Let $ G=(\mathbb{Z}_p)^k $, for $p>2$, and $ X $ a $(\textup{mod}\,p)$-cohomology $n$-sphere such that $ X^G=\emptyset $
		and $ X $ is a $ G $-ANR space. Then
		\begin{itemize}
			\item[a)] $ \ell(X)=
			\displaystyle\frac{n+1}{2}$.
			\item[b)] the Euler class $ e $ of $ X $ is polynomial.
		\end{itemize}	
	\end{corollary}
	\begin{proof}[Proof]
		For the item $ a) $, from Theorem \ref{theorem: length_cohom_sphere} or \cite[Proposition 3.6]{BMS} 
		we already have that $\frac{n+1}{2}\leq \ell(X)$.
		Now, since exists $ f\colon X\to SW $  such that
		$ \dim_{\mathbb{C}} SW=\dim X $ and
		$ \ell(SW)=\frac{n+1}{2}$, it follows that $\ell(X)\leq \ell(SW)
		\leq\frac{n+1}{2}$. This implies that there exists
		an polynomial element $\alpha \in (e)\cap P^*(G)$
		with same degree as $e$. Then $(\alpha)=(e)$ and part $b)$ follows.
	\end{proof}

%
\section{An upper-bound for the length and Bourgin-Yang theorems}
	As mentioned in the introduction, we will study an upper-bound for the length,
	first considering any compact Lie group, and then we specialize to the $p$-torus
	and torus case to obtain a Bourgin-Yang type result.
	Let $G$ be a compact Lie group and $X$ a compact $G$-space.
	In \cite{Segala,Segal} one can find, in the context
	of equivariant $K$-theory $K^*_G$, a filtration for $K^*_G(X)$. 
	The construction can be easily adapted for the Borel cohomology theory and
	is given as follows. Let $R$ be a commutative ring with unity.
	
	For any finite $G$-closed cover $\mathcal{U}$ of $X$, let $N_\mathcal{U}$ be the
	nerve of such cover. We can associate a $G$-compact space 
	$X_\mathcal{U}=\bigcup_{\sigma\in N_\mathcal{U}}(U_\sigma\times|\sigma|)
	\subset X\times |N_\mathcal{U}|$, where 
	$U_\sigma=\bigcap_{\alpha\in\sigma}U_\alpha\neq \emptyset$ and
	$|N_\mathcal{U}|$ is the geometric realization of $N_\mathcal{U}$. Let 
	$X^p_\mathcal{U}=\bigcup_{\dim(\sigma)\leq p}(U_\sigma\times|\sigma|)$. 
	This gives a filtration of $G$-subspaces 
	$X^0_\mathcal{U}\subset X^1_\mathcal{U}\subset X^2_\mathcal{U}\subset\cdots X_\mathcal{U}$.
	We say that an element of $H^*_G(X;R)$ is in $H^*_{G,s}(X;R)$ if, for some finite $G$-closed
	cover $\mathcal{W}$ of $X$, the element is in $\ker[H^*_G(X;R)\to
	H^*_G(X_\mathcal{W}^{s-1};R)]$. 
		
	\begin{lemma}\label{lemma: segal_construction}
		Let $G$ be a compact Lie group and $X$ a compact $G$-space. For any finite
		$G$-closed cover $\mathcal{U}$ of $X$:
		\begin{itemize}
			\item[a)] The projection on the first coordinate $\pi_1\colon X_{\mathcal{U}}
			\to X$ induces an isomorphism $\pi_1^*\colon H_G^*(X_{\mathcal{U}};R)
			\to H_G^*(X;R)$.
			
			\item[b)] If $\mathcal{V}$ is a refinement of $\mathcal{U}$, there exists
			a $G$-equivariant map $X_\mathcal{V}\to X_\mathcal{U}$, defined
			up to $G$-homotopy, that respects the filtrations and the projections onto X.
			Thus, $H^*_G(X;R)=H^*_{G,0}(X;R)\supseteq H^*_{G,1}\supseteq 
			\cdots\supseteq H^*_{G,s}(X;R)\cdots$.
			
			\item[c)] $H^*_{G,s}(X;R)\cdot H^*_{G,s'}(X;R)\subset H^*_{G,s+s'}(X;R)$,
			where $``\cdot"$ represents the multiplication between rings.
			\item[d)] $H^*_{G,1}(X;R)=\displaystyle\bigcap_{x\in X}
			\ker[H^*_G(X;R)\to H^*_G(G/G_x;R)]$.
		\end{itemize}
	\end{lemma}
	\begin{proof}
		For items a), b),d) we refer to \cite{Segala}, and for d) to \cite{Segal}.
	\end{proof}
	
	Let us suppress the ring of coefficients from our notation and keep the standards choices
	when we specialize to the $p$-torus cases.

	\begin{remark}\label{remark: dim_X}
		If $\dim(X)=n<\infty$ (covering dimension), then $H^*_{G,\dim X+1}(X)=0$. Indeed,
		there is a cover $\mathcal{U}$ of X such that  $X_\mathcal{U}=X^m_\mathcal{U}$,
		for all $m\leq\dim X$. Thus, by Lemma \ref{lemma: segal_construction} a), 
		$H^*_{G,\dim X +1}=\ker[H^*_G(X)\to
		H^*_G(X_\mathcal{U})]=\{0\}$.		
	\end{remark}

	Now we present some relations of this construction with the length.
	
	\begin{corollary}\label{corollary: segal_length}
		Let $G$ be a compact Lie group and $X$ a compact $G$-space.
		Consider $\mathcal{H}=\{H_\gamma\}_{\gamma\in \Gamma}$ the collection of all maximal
		isotropy subgroups in $G$, i.e., if $x\in X^{H_\gamma}$, we have $G_x=H_\gamma$,
		for all $\gamma\in\Gamma$. Then $H^*_{G,1}(X)=\bigcap_{\gamma\in\Gamma}
		\ker[H^*_G(X)\to H^*{(BH_\gamma)}]$.  
	\end{corollary}
	\begin{proof}
		Given $ x\in X $, if $G_x$ is not a maximal isotropy
		subgroup, then there exists a $\gamma\in\Gamma$ such that 
		$G_x<H_\gamma $. Then  $ G/G_x\to G/H_\gamma $ and,
		we conclude that, $ \ker [H^*_G(X)\to H^*(BH_\gamma)]\subset
		\ker [H^*_G(X)\to H^*_G(G/G_x)] $.
		Thus $\bigcap_{\gamma\in \Gamma}\ker [H^*_G(X)\to H^*_G(BH_\gamma)]
		\subseteq \bigcap_{x\in X}\ker [H^*_G(X)\longrightarrow H^*_G(G/G_x)]$.
		On the other hand, we have that		
		$\bigcap_{x\in X}\ker [H^*_G(X)\longrightarrow H^*_G(G/G_x)]\subseteq 
		\bigcap_{\gamma\in \Gamma}\ker [H^*_G(X)\to H^*(BH_\gamma)]$.
	\end{proof}
	
	\begin{theorem}\label{theorem: length_ubound}
		Let $G$ be a compact Lie group and $X$ a compact $G$-space.
		Suppose that $X^G=\emptyset$ and the number $\alpha(X)$ of maximal isotropy subgroups of  
		$X$	is finite. Then $\ell(X)\leq \alpha(X)\cdot(\dim X+1).$
	\end{theorem}
	\begin{proof}
		Let us say that $\mathcal{H}=\{H_i\}_{i\in\Gamma}$ is the collection
		of all maximal isotropy subgroups of $X$ and $|\Gamma|=\alpha(X)=s$.
		For any $ x\in X $, the composition $ G/G_x\hookrightarrow X
		\stackrel{p_{X}}{\to} \{\textup{pt}\} $ induces the diagram
		\begin{displaymath}
			\xymatrix{& H^*(BG)\ar[ld]_{p^*_{X}} \ar[rd] & \\
			H^*_G(X) \ar[rr]^{f_x} & & H^*(BG_x).}
		\end{displaymath}
		Given $w\in \bigcap_{i=1}^s\ker [H^*(BG)\to H^*(BH_i)]$,
		$ f_x\circ p^*_{X}(w) =0$, for all $ x\in X $ and, thus
		$ p^*_{X}(w)\in \bigcap_{x\in X} 
		\ker [H^*_G(X)\to H^*(BG_x)]=
		\bigcap_{i=1}^s\ker [H^*_G(X)\to H^*_G(BH_i)]=
		H^*_{G,1}(X)$. By Remark \ref{remark: dim_X},
		we have $p_X^*(w)^{\dim X+1}\in H^*_{G,\dim X+1}(X)=\{0\}$ and then
		$\omega^{\dim X+1}\cdot 1_X=0\in H^*_G(X)$.
		Considering $ \omega_i\in \ker[H^*_G(BG)\to H^*_G(BH_i)]$, for
		$ i=1,\ldots,s $, then $\omega=\omega_1\ldots\omega_{s}\in  
		\bigcap_{i=1}^s\ker [H^*(BG)\to H^*(BH_i)]$. 
		So $p^*_X(\omega^{\dim X+1})=0$ which implies $\ell(X)\leq s(\dim X+1)$.
	\end{proof}

	\begin{remark}\label{remark: dim_best}
		As commented in \cite[Example 6.1]{BMS}, one cannot expect better upper bound for
		this choice of length. For a family of distinct maximal subtori $H_i\subset G=(\mathbb{Z}_p)^k$, we have that for $X=\sqcup_{i=1}^t G/H_i$, $\alpha(X)=t$
		and $\dim(\sqcup_{i=1}^t G/H_i)=\dim(X)=0$, so $\ell(X)=t=\alpha(X)(\dim(X)+1)$.
		In the case $G=\mathbb{Z}_p$ we
		should have $\ell(X)\leq \dim X+1$.	Also,  by choosing a different
		collection for $\mathcal{A}$,
		$\widehat{\mathcal{A}}=\{A_1\sqcup \cdots\sqcup A_t;A_i\in\mathcal{A},\,t\geq 1\}$,
		the family of all finite disjoint union of orbit spaces in the definition of length,
		we  have $\ell_{\widehat{\mathcal{A}}}(X)\leq \dim X+1$.
	\end{remark}

	\begin{theorem}[Bourgin-Yang]\label{theorem: by1}
		Let $ G = (\mathbb{Z}_p)^k $ or $ (S^1)^k $,
		$ X$ a \linebreak$(\textup{mod}\,p)$-\-cohomology $ n $-sphere and 
		$ Y $ a $ G $-space, where $ X^G=\emptyset $  and
		$ Y-Y^G$ is a $(\textup{mod}\,p)$-cohomology $ m $-sphere.
		Given a $ G $-equivariant map  $ f\colon X\to Y $ and
		con\-sidering  $ Z_f=f^{-1}(Y^G) $, then the number $ \alpha=\alpha(Z_f)$ of subtori
		$ H\subset G$ of rank $ k-1 $ such that
		$X^H\subset Z_f\neq \emptyset $ is nonzero and
		$$\dim Z_f\geq  
		\begin{cases} 
			\frac{n-m}{\alpha}-1, \,\,\displaystyle\textrm{ if } p=2,\\
			\frac{n-m}{2\alpha}-1, \textrm{ if }  p=0 \textrm{ or }
			p>2 \textrm{ and } e(Y-Y^G)   \textrm{ is polynomial.}
		\end{cases}$$
		In particular, if $ n>m $, there is no 
		$ G $-equivariant map $ X\to Y-Y^G $.
	\end{theorem}
	
	\begin{proof}
		Specializing \cite[Theorem 3.1]{BMS} to our case yields
		 $\ell(Z_f)\geq \ell(X)-\ell(Y-Y^G)$. Now,
		combining  Theorem \ref{theorem: length_cohom_sphere} and 
		\ref{theorem: length_ubound} gives us the desired inequalities.
		
		 In the case $n>m$, we shall have $\dim(Z_f)\geq 0$ and, therefore $Z_f
		 \neq \emptyset$ which implies the non-existence of a $G$-equivariant map
		 between $X$ and $Y-Y^G$.
	\end{proof}

	\begin{remark}\label{remark: by_remark}
		When $k=1$, we obtain a better estimate for $p\neq 2$. Indeed,
		\cite[Theorem 3.1]{BMS} and  Theorem \ref{theorem: length_cohom_sphere}
		gives us $\ell(Z_f)\geq \frac{n-m}{2}$  (or $\ell(Z_f)\geq n-m$ for $p=2$) 
		which implies that
		$\bigoplus_{i=0}^{m-n-1}H^i(BG)\to H^*_G(Z_f)$
		is a monomorphism. Now, if the $\mathbb{Z}_p$-action is free, we conclude
		that $\textrm{cohom.dim}(Z_f)\geq n-m-1$, where ``cohom.dim" stands for cohomological dimension. This conclusion is a particular
		case found in the literature \cite{Dold,Izy,Jaw81a,Nakaoka}. The
		hypothesis that $e(Y-Y^G)$ is polynomial can be removed (see Remark \ref{remark: alaborel}).
	\end{remark}
	
%
\subsection{Bourgin-Yang theorem for topological manifolds}
	
	A version that offers an optimal estimate for the Bourgin-Yang theorem
	in the context of cohomology spheres could be obtained when we add to X,
	the domain of the $G$-equivariant map, the hypothesis that it is also a closed
	and orientable topological manifold. In fact, the result will be stated in a more
	general setting, by supposing that $X$ is a closed orientable topological manifold
	and $(n-1)$-acyclic, i.e., $H^i(X,\mathbb{F})=0$, when $1<i<n-1$, over
	the field $\mathbb{F}$ corresponding to $p$.
	
	\begin{theorem}
		Let $G=(\mathbb{Z}_p)^k$ or $(S^1)^k$ and $X,Y$ two $G$-spaces where:
		\begin{itemize}
			\item[i)] $X$ is a closed orientable topological manifold such that
			$H^i(X)=0$ for $1<i<n-1$.
			\item[ii)] $Y$ is $G$-CW complex and there exists $A\subset Y$ such that $Y-A$ is compact 
			(or paracompact with finite converging dimensional) and $H^i(Y-A)=0$
			for $i\geq m$. Additionally $(Y-A)^G=\emptyset$.
			
			\noindent
			In the case $p=0$, suppose that $Y$ has finitely many orbit type.
		\end{itemize}
		If $f\colon X\to Y$ is a $G$-equivariant map, we have 
		$\dim(f^{-1}(A))\geq n-m-1$.
	\end{theorem}
	\begin{proof}
		Let $Z:=f^{-1}(A)$. By hypothesis, $f|_{X-Z}:X-Z\to Y-A$ is a 
		$G$-equivariant map. Suppose by contradiction that $\dim(Z)<n-m-1$.
		
		Thus $H^i(Z)=0$, for $i\geq n-m-1$ and, from Alexander-Poincar\'{e}-Lefschetz duality,
		$0=H^i(Z)=H_{n-m-i}(X,X-Z)$. From the hypothesis on the homology of $X$ and the long
		exact sequence of the pair $(X,X-Z)$ we obtain that $0=\tilde{H}^i(Z)=
		H_{n-1-i}(X,X-Z)=\tilde{H}_{n-2-i}(X-Z),$ for $n-2-i\leq m-1$.
		
		Then we have that $0=\tilde{H}_q(X-Z)=\tilde{H}^q(X-Z)$, for $q<m$, and, that
		$0=H^q(Y-A)=H_q(Y-A)$, for $q\geq m$. It follows from 
		\cite[Theorem 6.4]{CP1} there is no $G$-equivariant from $X-Z$ to $Y-A$, which
		is a contradiction.		
	\end{proof}
	
	\begin{remark}
		This result is an immediate generalization of \cite[Theorem 2.5]{DEM}.
	\end{remark}

%
\section{General remarks about the length for $p$-torus group}

	Let us consider the following result proved in \cite[Proposition 3.1]{FH1}. Here, we only 
	state it for $p$-tori groups.

	\begin{proposition}\label{proposition: FH_prop3.1}
		Let $G=(\mathbb{Z}_p)^k$ or $(S^1)^k$ and $ G_1$,  $G_2 $ two subtorus  such that 
		$G=G_1\times G_2$.	Consider a $ G_i $-equivariant map 
		$ p_{X_i}\colon X_i\to B_i $, for $i=1,2$, and a $ G $-equivariant map
		$p_{X}\colon  X\to B $,  where $X=X_1\times X_2$, $B=B_1\times B_2$ are
		$G$-spaces through diagonal actions and $p_X=(p_{X_1},p_{X_2})$. Assume that all 
		spaces are paracompact. Then we have 
		$\ker p^*_X =\ker(p^*_{X_1})\otimes_R H^*_{G_2}(B_2) 
		+ H^*_{G_1}(B_1)\otimes_R \ker(p^*_{X_2})$.
	\end{proposition}
	
	\begin{corollary}\label{corollary: FH_prop3.1}
		Let $G=(\mathbb{Z}_p)^k $ or $(S^1)^k $	and $X$ a paracompact $G$-space.
		For a subtorus $H$  of rank $k-1$ of $G$ such that $X^H\neq\emptyset$, we have that
		$\ker (p^*_{X^H})= H^*(BH)\otimes\, \ker (q^*_{X^H})$, $\,$
		where $ p_{X^H}\colon X^H\longrightarrow \{\textup{pt}\} $ is a 
		$G$-equivariant map and  $ q_{X^H}\colon X^H\longrightarrow \{\textup{pt}\} $ is
		a  $ L\cong G/H $-equivariant map. 
	\end{corollary}
	\begin{proof}
		The result follows from Proposition \ref{proposition: FH_prop3.1} considering 
		the spaces	$ G_1=H$, $G_2 = L \cong G/H$, $X_2=X^H$, $X_1=B_1=B_2=\{\textup{pt}\}$
		and	the equivariant maps $p_{X_1}=q_{X^H}$ and $p_{X_2}=p_{X^H}$.
	\end{proof}

	Now we state a lower bound for the length $\ell$ in terms of the 
	length $ \ell_H $ of subtori of rank $k-1$.

	\begin{theorem}\label{theorem: length_lbound}
		Let $ G= (\mathbb{Z}_p)^k $ or $(S^1)^k $ and  $ X $ a compact $ G $-space sucht that 
		$ X^G=\emptyset $. We have $\sum_{H\in\mathcal{H}} \ell_H(X^H)\leq \ell(X)$,
		where $ \mathcal{H} $ is the collection of all subtori of rank $ k-1 $ of $G$
		such that $ X^H \neq \emptyset $.
	\end{theorem}
	\begin{proof}[Proof]
		Let us consider the case $p>2$. Fix $ H $ 	a subtorus of rank $ k-1 $. Then,
		$ G/H\to X^H\hookrightarrow X\to\{\textup{pt}\} $
		induces the following commutative diagram.
		
		\begin{displaymath}
		\xymatrix{
			H^*(BG)\ar[r]^{i_H^*}\ar[d]^{p^*_X}\ar[rd]^{p^*_{X^H}} & H^*(BH)\\
			H^*_G(X)\ar[r]& H^*_G(X^H)\ar[u]}
		\end{displaymath}
		
		Thus, $ \ker p^*_X\subseteq \ker p^*_{X^H}\subseteq\ker i^*_H$ and, therefore,
		$I\cap\ker p^*_X\subseteq I\cap\ker p^*_{X^H}\subseteq I\cap\ker i^*_H=(s_H)$,
		where $I=P^*(G)\cong\mathbb{Z}_p[t_1,\ldots,t_k]$. This implies that
		$I\cap \ker{p_{X^H}^*}=(s_H^b)$, where $b=\ell_H(X^H)$. Indeed:
		
		From Corollary \ref{corollary: FH_prop3.1} we have that
		$ \ker p^*_{X^H}\cong H^*(BH;\mathbb{Z}_p)\otimes_{\mathbb{Z}_p}\ker{q^*_{X^H}}$,
		where the $ G/H $-equivariant map	
		$ q_{X^H}\colon X^H\to\{\textup{pt}\} $ induces 
		$ q^*_{X^H}\colon H^*(B(G/H))\to H^*_{G/H}(X^H)$.
		Note that, since $G/H\cong\mathbb{Z}_p$, we should have 
		$P^*(G/H)\cap\ker{q^*_{X^H}}= (t^b)$, where $t\in \mathbb{Z}_p[t]\otimes \Lambda(s)
		\cong H^*(B(G/H))$. Then, 
		$I\cap\ker p^*_{X^H}\cong I\cap (H^*(BH;\mathbb{Z}_p)\otimes_{\mathbb{Z}_p}\ker{q^*_{X^H}})
		\cong P^*(H)\otimes(t^r)\cong (s_H^b)$. Thus $ I\cap\ker p^*_X\subseteq(s_H^b) $.
		Considering the definition of the length, 
		we should have $\ell_H(X^H)=\min\{\lambda;\; p_{X^H}(t^\lambda)=0\}=b$.
		
		Assuming that
		$\mathcal{H}=\{H_1,\ldots,H_s\}$, we should have that 
		$I\cap\ker p^*_X\subseteq(s_{H_1}^{b_1})\cap\cdots\cap (s_{H_s}^{b_s})$,
		where $\ell_{H_i}(X^{H_i})=b_i$. Now, given
		an element $ z\in I\cap \ker p^*_X $, this corresponds to an polynomial element
		multiple of the polynomials $s_{H_1}^{b_1}\cdots s_{H_s}^{b_s}$ in $ H^*(BG) $,
		which implies that $\ell(X)\geq \sum_{i=1}^s\ell_{H_i}(X^{H_i})$.
		
		For $ p=0 $ or $2 $, we carry out the proof in the same way and we do not need to care
		about taking intersections with the polynomial ring part.
	\end{proof}

	\begin{remark}\label{remark: alaborel}
		We shall have the equality {\it \`{a} la} Borel Formula, 
		$ \ell(X)=\sum_{H\in\mathcal{H}} \ell_{H}(X^H) $ if, and only if,
		the kernel $ \ker p^*_X $ contains the polynomial element  
		 $s_{H_1}^{b_1}\cdots s_{H_s}^{b_s}$.
		 
		 The hypothesis that $e(Y-Y^G)$ is polynomial in Theorem \ref{theorem: by1}
		 can be removed. Indeed, considering the statement and notations there, we can work with $\ell(Z_f)\geq \sum_H \ell_H(Z_f^H)\geq
		 \sum_H [\ell(X^H)-\ell_H((Y-Y^G)^H)]$ and here $e((Y-Y^G)^H)$ will be polynomial.
	\end{remark}
	
	\textbf{Acknowledgments.} The first author was supported by FAPESP under the grant
	2011/23610-3. The authors would like to thank Wac\l{}aw Marzantowicz for proposing
	us to work over generalizations of theorems regarding cohomology spheres and
	equivariant maps.

{}


\begin{thebibliography}{99}
		
		\bibitem{BartschClapp}
		Bartsch, T., Clapp, M.: 
		\emph{Bifurcation theory for symmetric potential
			  operators and the equivariant cup-length}. 
		Math. Z. \textbf{204}, 341-356 (1990).
		
		\bibitem{Bartsch}	
		Bartsch, T.: \textit{Topological Methods for Variational Problems with Symmetries}. 
		Lecture Notes in Mathematics \textbf{1560}, Springer-Verlag Berlin Heidelberg (1993).
		
		\bibitem{BMMS}
		B\l{}aszczyk, Z., Marzantowicz, W., Singh, M.:
		\emph{Equivariant maps between representation spheres}.
		Bull. Belg. Math. Soc. Simon Stevin
		\textbf{24}, N. 4, 621-630  (2017).
		
		\bibitem{BMS}
		B\l{}aszczyk, Z., Marzantowicz, W., Singh, M.: \emph{General Bourgin-Yang theorems}.
		Topology and its Applications, \textbf{249}, 112-126 (2018).
	
		\bibitem{Borsuk}
		Borsuk, K.: \emph{Drei Sätze über die $ n $-dimensionale euklidische Sph\"{a}re}, Fund. Math. \textbf{20}, 177-190, (1933).
		
		\bibitem{Bourgin} 
		Bourgin, D. G.: \textit{On some separation and mapping theorems}. 
		Comment. Math. Helv. \textbf{29}, 199-214 (1955).
				
		\bibitem{CP1} 
		Clapp, M., Puppe, D.: \textit{Critical point theory with symmetries}. 
		J. Reine Angew. Math. \textbf{418}, 1-29  (1991).
		
		\bibitem{Deo} 
		Deo, S., Tripathi, S.: \textit{Compact Lie Group Action on Finitistic Spaces}. 
		Topology \textbf{21}, No. 4, pp. 393-399 (1982).
		
		\bibitem{Dieck}
		Dieck, T. T.: \textit{Transformation Groups}, de Gruyter, Berlin (1987).
		
		\bibitem{Dold} 
		Dold, A.: \textit{Parametrized Borsuk-Ulam theorems}. 
		Comment. Math. Helv. \textbf{63}, n. 2, 275-285  (1988).
		
		\bibitem{FH1}
		Fadell, E., Husseini, S. Y.: \textit{An ideal-valued cohomological index theory
		with applications to Borsuk-Ulam and Bourgin-Yang theorems}. 
		Ergodic Theory Dynam. Systems \textbf{8}, Charles Conley Memorial Issue, 73-85 (1988).
		
		\bibitem{Hsiang} 
		Hsiang, W. Y.: \textit{Cohomology Theory of Topological Transformation Groups}.
		Springer-Verlag, Berlin, Heidelberg, New York (1975).
				
		\bibitem{Izy} 
		Izydorek, M., Rybicki, S.: \textit{On parametrized Borsuk-Ulam theorem for free 
		$ \mathbb{Z}_p $-action}. Algebraic topology (San Feliu de Guixols 1990) 227-234, 
		Lecture Notes in Mathematics \textbf{1509}, Springer, Berlin, (1992).
		
		\bibitem{Jaw81a} 
		Jaworowski, J.: \emph{A continuous version of the  Borsuk-Ulam theorem}. 
		Proc. Amer. Math. Soc., \textbf{82}(1):112-114, (1981).		
		
		\bibitem{MAR1} 
		Marzatowicz, W.: \textit{Borsuk-Ulam theorem for any compact Lie groups}. 
		J. Lond. Math. Soc. \textbf{49}, 195–208 (1994).
		
		\bibitem{MarzIzy2} 
		Marzatowicz, W., Izydorek, M.: \textit{The Borsuk-Ulam properties for cyclic groups}. 
		Topol. Methods in Nonlinear Anal. \textbf{516}, 65-72 (2000).
		
		\bibitem{DEM1}
		Marzantowicz, M., de Mattos, D., dos Santos, E. L.: 
		\textit{Bourgin-Yang version of the Borsuk-Ulam theorem for 
		$ \mathbb{Z}_{p^k} $ -equivariant maps}. 
		Algebraic and Geometric Topology \textbf{12}, 2146-2158 (2012) .
		
		\bibitem{DEM}
		Marzantowicz, M., de Mattos, D., dos Santos, E. L.: 
		\textit{Bourgin-Yang versions of the Borsuk-Ulam theorem for $ p $-toral groups}. 
		Journal of Fixed Point Theory and Applications, \textbf{19}, p.1427-1437, (2017).
		
		\bibitem{MUR}
		Murayama, M.: On $G$-ANRs and their $G$-homotopy type. textit{Osaka J. Math.}
		\textbf{20}, 479-512 (1983).
			
		\bibitem{Nakaoka} 
		Nakaoka, M.: \textit{Parametrized Borsuk-Ulam theorems and characteristic polynomials}.
		Topological fixed point theory and applications (Tianjin, 1988), 155-170. 
		Lecture Notes in Mathematics \textbf{1411}, Springer, Berlin, (1989).
		
		\bibitem{Segala} 
		Segal, G.: \textit{Classifying Spaces and Spectral Sequences}. 
		Publications  Math\'{e}matiques de l'Institut des Hautes \'{E}tudes Scientifiques
		January , Volume \textbf{34}, Issue 1, pp 105-112 (1968).
		
		\bibitem{Segal}
		Segal, G.: \textit{Equivariant $ K $-theory}. 
		Publ. Math. IHES \textbf{34}, 129-151 (1968).	
		
		\bibitem{Yang} 
		Yang, C. T.: 
		\textit{On the theorems of Borsuk-Ulam, Kakutani-Yamabe-Yujobo and Dyson, I}.
		Ann. of Math. \textbf{60}, 262-282 (1954).
		
		\bibitem{Yang2} 
		Yang, C. T.: \textit{On Theorems of Borsuk-Ulam, Kakutani-Yamabe-Yujobo and Dyson, II}.
		Annals of Mathematics, Second Series, \textbf{62}, No. 2 (Sep., 1955), pp. 271-283.
\end{thebibliography}
\end{document}